\documentclass[11pt]{article}

\usepackage{amssymb,amsmath,amsfonts,amsthm}
\usepackage{latexsym}
\usepackage{graphics}
\usepackage{indentfirst}

\newtheorem{innercustomthm}{Theorem}
\newenvironment{customthm}[1]
  {\renewcommand\theinnercustomthm{#1}\innercustomthm}
  {\endinnercustomthm}

\newtheorem{innercustomcor}{Corollary}
\newenvironment{customcor}[1]
  {\renewcommand\theinnercustomcor{#1}\innercustomcor}
  {\endinnercustomcor}
\newtheorem*{thm*}{Theorem}
\newtheorem{thm}{Theorem}
\newtheorem{lem}[thm]{Lemma}
\newtheorem{pro}[thm]{Proposition}

\newtheorem{cor}[thm]{Corollary}
\newtheorem{conj}[thm]{Conjecture}

\newtheorem{ques}[thm]{Question}

\newcommand{\N}{\mathbb{N}}

\newcommand{\floor}[1]{\left\lfloor#1\right\rfloor}
\newcommand{\ceil}[1]{\left\lceil#1\right\rceil}

\begin{document}

\title{Proportional Choosability: A New List Analogue of Equitable Coloring}

\author{Hemanshu Kaul\footnotemark[1], Jeffrey A. Mudrock\footnotemark[1], Michael J. Pelsmajer\footnotemark[1], and Benjamin Reiniger\footnotemark[1]}

\footnotetext[1]{Department of Applied Mathematics, Illinois Institute of Technology, Chicago, IL 60616.  E-mail:  {\tt {kaul@iit.edu, jmudrock@hawk.iit.edu, pelsmajer@iit.edu, breiniger@iit.edu.}}}

\date{2018}

\maketitle

\begin{abstract}
In 2003, Kostochka, Pelsmajer, and West introduced a list analogue of
equitable coloring called equitable choosability.  In this paper, we
motivate and define a new list analogue of equitable coloring called
proportional choosability.  A $k$-assignment $L$ for a graph $G$
specifies a list $L(v)$ of $k$ available colors for each vertex $v$ of
$G$.  An $L$-coloring assigns a color to each vertex $v$ from its list
$L(v)$.  For each color $c$, let $\eta(c)$ be the number of vertices
$v$ whose list $L(v)$ contains $c$.  A {\em proportional $L$-coloring}
of $G$ is a proper $L$-coloring in which each color $c \in \bigcup_{v \in
V(G)} L(v)$ is used $\lfloor \eta(c)/k \rfloor$ or $\lceil \eta(c)/k
\rceil$ times.  A graph $G$ is \emph{proportionally $k$-choosable} if
a proportional $L$-coloring of $G$ exists whenever $L$ is a
$k$-assignment for $G$.  We show that if a graph $G$ is proportionally
$k$-choosable, then every subgraph of $G$ is also proportionally
$k$-choosable and also $G$ is proportionally $(k+1)$-choosable, unlike
equitable choosability for which analogous claims would be false.  We
also show that any graph $G$ is proportionally $k$-choosable whenever
$k \geq \Delta(G) + \lceil |V(G)|/2 \rceil$, and we use matching
theory to completely characterize the proportional choosability of
stars and the disjoint union of cliques.   

\medskip

\noindent {\bf Keywords.} graph coloring, equitable coloring, list coloring, equitable choosability.

\noindent \textbf{Mathematics Subject Classification.} 05C15.

\end{abstract}

\section{Introduction}\label{intro}

In this paper we define and study a new list analogue of equitable coloring.  All graphs are assumed to be finite, simple graphs unless otherwise noted.  Generally speaking we follow West~\cite{W01} for basic terminology and notation.

\subsection{Equitable Coloring}


The study of equitable coloring began with a conjecture of Erd\H{o}s~\cite{E64} in~1964 (Theorem~\ref{thm: HS}, proved in 1970~\cite{HS70}), but the general concept was formally introduced by Meyer~\cite{M73} in~1973.  An \emph{equitable $k$-coloring} of a graph $G$ is a proper $k$-coloring of $G$, $f$, such that the sizes of the color classes differ by at most one (where a $k$-coloring has exactly $k$, possibly empty, color classes).  It is easy to see that for an equitable $k$-coloring, the color classes associated with the coloring are each of size $\lceil |V(G)|/k \rceil$ or $\lfloor |V(G)|/k \rfloor$.  We say that a graph $G$ is \emph{equitably $k$-colorable} if there exists an equitable $k$-coloring of $G$. Equitable colorings are useful when it is preferable to form a proper coloring without under-using any colors or using any color excessively often. Equitable coloring has found many applications ( see for example~\cite{T73,P01,KJ06,JR02}).

Unlike ordinary graph coloring, increasing the number of colors may make equitable coloring more difficult.  For example, $K_{3,3}$ is equitably 2-colorable, but it is not equitably 3-colorable.  More generally, $K_{2m+1,2m+1}$ is equitably $k$-colorable for each even $k$ less than $2m+1$, it is not equitably $(2m+1)$-colorable, and it is equitably $k$-colorable for each $k \geq 2m+2 = \Delta(K_{2m+1,2m+1})+1$ where we use $\Delta(G)$ to denote the largest degree of a vertex in $G$ (see~\cite{LW96} for further details).  In 1970, Hajn\'{a}l and Szemer\'{e}di proved the following general result.

\begin{thm}[\cite{HS70}] \label{thm: HS}
Every graph $G$ has an equitable $k$-coloring when $k \geq \Delta(G)+1$.
\end{thm}

In 1994, Chen, Lih, and Wu~\cite{CL94} made a conjecture in the spirit of Brooks's Theorem.  Their conjecture is known as the $\Delta$-Equitable Coloring Conjecture ($\Delta$-ECC for short).

\begin{conj}[\cite{CL94}, {\bf $\Delta$-ECC}] \label{conj: ECC}
A connected graph $G$ is equitably $\Delta(G)$-colorable if it is different from $K_m$, $C_{2m+1}$, and $K_{2m+1,2m+1}$.
\end{conj}

Conjecture~\ref{conj: ECC} has been proven true for interval graphs, trees, outerplanar graphs, subcubic graphs, and several other classes of graphs (see~\cite{CL94}, \cite{L98}, and~\cite{YZ97}).

\subsection{Equitable Choosability}

In this subsection, we briefly review a well known list coloring analogue of equitable coloring. This will motivate a new list analogue of equitable coloring which we will present afterward and will be the focus of this paper. 

In 2003, Kostochka, West, and the third author introduced a list version of equitable coloring~\cite{KP03}, equitable choosability, which has received some attention in the literature.  A list assignment $L$ for a graph $G$ is a function which associates with each vertex of $G$ a list of colors.  When $|L(v)|=k$ for all $v \in V(G)$ we say that $L$ is a \emph{k-assignment} for $G$.  We say $G$ is \emph{$L$-colorable} if there exists a proper coloring $f$ of $G$ such that $f(v) \in L(v)$ for each $v \in V(G)$ (and $f$ is called a \emph{proper $L$-coloring} of $G$).  We say $G$ is \emph{$k$-choosable} if a proper $L$-coloring of $G$ exists whenever $L$ is a $k$-assignment for $G$.

Suppose $L$ is a $k$-assignment for the graph $G$.  A proper $L$-coloring of $G$ is \emph{equitable} if each color appears on at most $\lceil |V(G)|/k \rceil$ vertices.  Such a coloring is called an \emph{equitable $L$-coloring} of $G$, and we call $G$ \emph{equitably $L$-colorable} when an equitable $L$-coloring of $G$ exists.  We say $G$ is \emph{equitably $k$-choosable} if $G$ is equitably $L$-colorable whenever $L$ is a $k$-assignment for $G$.  Note that $\lceil |V(G)|/k \rceil$ is the same upper bound on the size of color classes in typical equitable coloring.  So, for the notion of equitable choosability, no color is used excessively often.

It is conjectured in~\cite{KP03} that Theorem~\ref{thm: HS} and the $\Delta$-ECC hold in the list context.

\begin{conj}[\cite{KP03}] \label{conj: KPW1}
Every graph $G$ is equitably $k$-choosable when $k \geq \Delta(G)+1$.
\end{conj}

\begin{conj}[\cite{KP03}] \label{conj: KPW2}
A connected graph $G$ is equitably $k$-choosable for each $k \geq \Delta(G)$ if it is different from $K_m$, $C_{2m+1}$, and $K_{2m+1,2m+1}$.
\end{conj}

In~\cite{KP03} it is shown that Conjectures~\ref{conj: KPW1} and~\ref{conj: KPW2} hold for forests, connected interval graphs, and 2-degenerate graphs with maximum degree at least 5.  Conjectures~\ref{conj: KPW1} and~\ref{conj: KPW2} have also been verified for outerplanar graphs~\cite{ZB10}, series-parallel graphs~\cite{ZW11}, and certain planar graphs (see~\cite{LB09}, \cite{ZB08}, and~\cite{ZB15}).  In 2013, Kierstead and Kostochka made substantial progress on Conjecture~\ref{conj: KPW1}, and proved it for all graphs of maximum degree at most 7 (see~\cite{KK13}).

\subsection{Proportional Choosability}

If a graph is $k$-choosable, then it is $k$-colorable.  However, it can happen that a graph is equitably $k$-choosable, but not equitably $k$-colorable.  For example, $K_{1,6}$ is equitably 3-choosable, but it is not equitably 3-colorable.  This shows that the notion of equitable choosability is not actually a strengthening of equitable coloring.  Generally speaking, the reason a graph can be equitably $k$-choosable without being equitably $k$-colorable is that in an equitable $k$-coloring we are forced to not under-use any colors, but for an equitable list coloring there is no lower bound on how many times a color must be used.

In some sense it is reasonable to not place any lower bound on the number of times a color should be used in the list context, since colors in a list assignment might only appear on a few lists, which of course  limits the number of times that the color will appear in a list coloring.  The number of lists in which a color appears could be less than $\lfloor |V(G)|/k \rfloor$ or even as low as one.  Equitable choosability generalizes equitable coloring from the side where it is possible to do so, and it abandons trying to control it from the other side.

In this paper we present a way to fix this difficulty.  In particular, we introduce the notion of proportional choosability which will generalize equitable coloring from both sides. We will have that if a graph is proportionally $k$-choosable, it must be equitably $k$-colorable.  The creation of this notion was facilitated by a question asked by Stasi~\cite{S17} during a talk given by the second author~\cite{M17}.

Suppose that $L$ is a $k$-assignment for graph $G$.  The \emph{palette of colors associated with $L$} is $\cup_{v \in V(G)} L(v)$.  For each color $c$, the \emph{multiplicity of $c$ in $L$} is the number of lists of $L$ in which $c$ appears.  The multiplicity of $c$ in $L$ is denoted by
$\eta_L(c)$ (or simply $\eta(c)$ when the list assignment is clear), so $\eta_L(c)=\left\lvert{\{v : v \in V(G), c \in L(v) \}}\right\rvert$.  Throughout this paper, when $L$ is a list assignment for some graph $G$, we always use $\mathcal{L}$ to denote the palette of colors associated with $L$.

Given a $k$-assignment $L$ for a graph $G$, a proper $L$-coloring, $f$, for $G$ is a \emph{proportional $L$-coloring} of $G$ if for each $c \in \mathcal{L}$, $f^{-1}(c)$, the color class of $c$, is of size
$$ \left \lfloor \frac{\eta(c)}{k} \right \rfloor \; \; \text{or} \; \; \left \lceil \frac{\eta(c)}{k} \right \rceil.$$
We say that graph $G$ is \emph{proportionally $L$-colorable} if a proportional $L$-coloring of $G$ exists, and $G$ is \emph{proportionally $k$-choosable} if $G$ is proportionally $L$-colorable whenever $L$ is a $k$-assignment for $G$.  The next proposition shows that this notion has the property that we want.

\begin{pro} \label{pro: motivation}
If $G$ is proportionally $k$-choosable, then $G$ is both equitably $k$-choosable and equitably $k$-colorable.
\end{pro}

\begin{proof}
Suppose $G$ is proportionally $k$-choosable.  When $L$ is a $k$-assignment for $G$, it is clear that a proportional $L$-coloring for $G$ is an equitable $L$-coloring for $G$ since $\eta(c) \leq |V(G)|$ for each $c \in \mathcal{L}$.  Thus, $G$ is equitably $k$-choosable.

To see that $G$ is equitably $k$-colorable, consider the list assignment $K$ that assigns the list $\{1,2, \ldots, k \}$ to each $v \in V(G)$.  There must be a proportional $K$-coloring, $f$, for $G$.  Clearly, $f$ is an equitable $k$-coloring of $G$.
\end{proof}

For disconnected graphs, equitable colorings on components can be merged after appropriately permuting color classes within each component~\cite{YZ97} to obtain an equitable coloring of the whole graph.  Equivalently, if $G_1,\ldots,G_t$ are pairwise vertex disjoint graphs,
each of which has an equitable $k$-coloring, then their disjoint union $\sum_{i=1}^t G_i$ also has an equitable $k$-coloring.
This property does not generalize to equitable choosability; we know that $K_{1,6}$ and $K_2$ are equitably 3-choosable but $K_{1,6} + K_2$ is not equitably 3-choosable. We will see below that it also does not hold in the context of proportional choosability. Thus, techniques that we use to find equitable $L$-colorings or proportional $L$-colorings on connected graphs may not be successful when it comes to disconnected graphs.

Before ending this subsection, there are a couple of historical remarks that are worth making.  Suppose that $G$ is a graph and $L$ is a list assignment for $G$.  Moreover, suppose that $p$ is a mapping that associates a positive integer $p(c)$ to each color $c \in \mathcal{L}$.  The decision problem of whether there is a proper $L$-coloring of $G$ such that each $c \in \mathcal{L}$ is used exactly $p(c)$ times, denoted $(G, L, p)$, has been studied before.  Determining whether a proportional $L$-coloring of $G$ exists is a slightly less restrictive instance of this problem.

In 1997, de Werra~\cite{W97} showed that when $G$ is a disjoint union of cliques, $(G, L, p)$ is in P.  Then, in 1999, Dror et al.~\cite{DF99} showed that $(P_n, L, p)$ is NP-complete even if $|L(v)| \leq 2$ for every vertex in the path.  However, if the palette of colors associated with $L$ must be of size at most $k$, then problem $(P_n, L, p)$ can be solved in time $O(n^k)$ by dynamic programming (see~\cite{DF99}).  Also, in 2002~\cite{GK02} it was shown that when $G$ is a planar bipartite graph and the palette of colors associated with $L$ is of size at most 3, $(G, L, p)$ is NP-complete.

\subsection{Results and Questions}

We now present an outline of the paper while summarizing our results and mentioning some open questions. In Section~\ref{easyresults} we present some initial results related to proportional choosability which will be utilized in subsequent sections.  Most importantly, we use matching theory to show that proportional choosability is {\it monotonic in $k$}, meaning that if $G$ is proportionally $k$-choosable, then it must be proportionally $(k+1)$-choosable as well.  We also show that proportional $k$-choosability is {\it monotone}, meaning that if $H$ is a subgraph of $G$ and $G$ is proportionally $k$-choosable, then $H$ is also proportionally $k$-choosable.

Those results are surprising, considering that equitable coloring and equitable list coloring do not behave so nicely.  Trivially, $k$-colorability and $k$-choosability are monotone and imply $(k+1)$-colorability and $(k+1)$-choosability. However, if a graph is equitably $k$-colorable (resp.~choosable), it need not be equitably $(k+1)$-colorable (resp.~choosable).  Indeed, $K_{3,3}$ is equitably 2-colorable and is not equitably 3-colorable, and $K_{1,9}$ is equitably 4-choosable and is not equitably 5-choosable.  Moreover, the graph property of being equitably $k$-colorable (resp.~choosable) is not monotone.  Indeed, $K_{3,3}$ is equitably 2-colorable, but $K_{1,3}$ is not equitably 2-colorable, and $K_{1,6}$ is equitably 3-choosable, but $K_{1,5}$ is not equitably 3-choosable.

The fact that we have monotonicity in $k$ when it comes to proportional choosability leads us to introduce a graph invariant.  In particular, for any graph $G$, the \emph{proportional choice number} of $G$, denoted $\chi_{pc}(G)$, is the smallest $k$ such that $G$ is proportionally $k$-choosable.  By monotonicity in $k$, we know that any graph $G$ is proportionally $k$-choosable if and only if $k \geq \chi_{pc}(G)$.

In Section~\ref{smallorder} we study the proportional choosability of graphs of small order. We give an algorithmic argument to convert an equitable $L$-coloring with some additional restrictions into a proportional $L$-coloring for a $k$-assignment $L$ of $G$ with every color having multiplicity less than $2k$ (Lemma~\ref{lem: algorithm}), which helps us prove the following result for proportional choosability in the spirit of an earlier result~(\cite[Theorem~1.1]{KP03}).

\begin{thm} \label{thm: smallorder}
Any graph $G$ with $\Delta(G) \geq 1$ satisfies
$$\chi_{pc}(G) \leq \Delta(G) + \left \lceil \frac{|V(G)|}{2} \right \rceil.$$
\end{thm}

By Corollary~\ref{cor: fullcliques} below, if $\Delta(G)=0$, then $\chi_{pc}(G)=1$.

It is natural to ask whether the proportional analogues of Conjectures~\ref{conj: KPW1} and~\ref{conj: KPW2} hold (i.e., can we replace ``equitably $k$-choosable" with ``proportionally $k$-choosable" in each of these conjectures).  We will show in Section~\ref{easyresults} that the proportional analogue of Conjecture~\ref{conj: KPW2} does not hold.  However, the proportional analogue of Conjecture~\ref{conj: KPW1} is open.

\begin{ques} \label{ques: fullHS}
For any graph $G$, is $G$ proportionally $k$-choosable whenever $k \geq \Delta(G)+1$?
\end{ques}

Another question that can be asked about the proportional choosability of graphs of small order is whether an analogue of Ohba's Conjecture holds (see~\cite{O02} and~\cite{NR15}). In Section~\ref{easyresults} we show that $\chi_{pc}(K_{2*m}) > m$ whenever $m \geq 2$, where
$K_{2*m}$ is the complete $m$-partite graph where each partite set is of size 2.  This leads to the following question.

\begin{ques} \label{ques: proOhba}
If $G$ is equitably $k$-colorable and $|V(G)| \leq 2k-1$, must it be that $G$ is proportionally $k$-choosable?
\end{ques}

Finally, in Section~\ref{matching} we use more matching theory to prove some further results on proportional choosability.  Specifically, we begin by proving a result for disconnected graphs.

\begin{thm} \label{thm: fullcomponents}
If $G$ is a graph such that its largest component has $k$ vertices, then $\chi_{pc}(G) \leq k$.
\end{thm}

The following corollary is related to the problem studied by de Werra~\cite{W97} mentioned at the end of the previous subsection.

\begin{cor} \label{cor: fullcliques}
If $G$ is a disjoint union of cliques and the largest component of $G$ has $t$ vertices, then $\chi_{pc}(G)=t$.
\end{cor}

We also completely characterize the proportional choosability of stars.

\begin{thm} \label{thm: fullstars}
$\chi_{pc}(K_{1,m}) = 1 + \lceil m/2 \rceil$.
\end{thm}

So $K_{1,m}$ is proportionally $k$-choosable if and only if $k \geq 1 + m/2$.  In terms of maximum degree, this is the same as the sharp bound known for equitable $k$-choosability of forests (see~\cite{KP03}).  On the other hand, we will see in Section~\ref{easyresults} that the proportional analogue of Conjecture~\ref{conj: KPW2} does not even hold for trees.

\section{Preliminary Results} \label{easyresults}

\subsection{Basic Results}

We begin by going over some terminology and simple results.  For the remainder of this paper, if $L$ is a list assignment for a graph $G$, and $G'$ is a subgraph of $G$, then we write $L' = L |_{V(G')}$ when $L'$ is the list assignment for $G'$ obtained by restricting $L$ to $V(G')$.


\begin{lem} \label{cor: extendstrong}
Let $G$ be a graph and $L$ a $k$-assignment for $G$. Let $S = \{x_1, \ldots, x_k \}$ so that $x_1, \ldots, x_k$ are distinct vertices in $G$.  Suppose that $L(x_i)$ is the same list for each $x_i \in S$.  If $G-S$ has a proportional $L'$-coloring where $L' = L|_{V(G)-S}$ and
$$|N_G(x_i) - S| \leq k-i$$
for $1 \leq i \leq k$, then $G$ has a proportional $L$-coloring.
\end{lem}

Lemma~\ref{cor: extendstrong} follows immediately from this next, more general result which has a straightforward proof that can be found in~\cite{M18}.

\begin{lem} \label{lem: extendstrong}
Let $G$ be a graph, $L$ a $k$-assignment for $G$, and $S \subseteq V(G)$.  Let $\mathcal{K} = \bigcup_{v \in  S} L(v)$.  Suppose that for each $a \in \mathcal{K}$ there is an $m_a \in \N$ so that
$$ |\{v \in S : a \in L(v)\}| = m_ak. $$
Suppose $f_1$ is a proportional $L'$-coloring of $G-S$ where $L'=L|_{V(G)-S}$.  Let $L''(v) = L(v) - \{f_1(u) : u \in (N_G(v) - S) \}$
for each $v \in S$.  If there is a proper $L''$-coloring, $f_2$, of $G[S]$ that uses each $a \in \mathcal{K}$ exactly $m_a$ times, then $f_1$ together with $f_2$ form a proportional $L$-coloring of $G$.
\end{lem}

Suppose that $L$ is a $k$-assignment for the graph $G$, and suppose that $f$ is a proper $L$-coloring of $G$.  We call $a \in \mathcal{L}$ a \emph{well distributed color in $L$} (or simply a \emph{well distributed color} when the list assignment is clear) if $\eta(a)$ is divisible by $k$.  A color $p \in \mathcal{L}$ is called \emph{perfectly used with respect to $f$} if $p$ is well distributed and $|f^{-1}(p)|= \eta(p)/k$.  A color $b \in \mathcal{L}$ is called \emph{almost excessive with respect to $f$} if $b$ is not well distributed and $|f^{-1}(b)| = \lceil \eta(b)/k \rceil$, and a color $d \in \mathcal{L}$ is called \emph{almost deficient with respect to $f$} if $d$ is not well distributed and $|f^{-1}(d)| = \lfloor \eta(d)/k \rfloor$.  A color $a' \in \mathcal{L}$ is called \emph{excessive with respect to $f$} if $|f^{-1}(a')| > \lceil \eta(a')/k \rceil$, and a color $d' \in \mathcal{L}$ is called \emph{deficient with respect to $f$} if  $|f^{-1}(d')| < \lfloor \eta(d')/k \rfloor$.

Throughout this paper whenever $L$ is a $k$-assignment for $G$ and $c\in \mathcal{L}$, write $\eta(c) = kq_c + r_c$ where $0 \leq r_c \leq k-1$. We can easily count almost excessive colors by adding up all the remainders, as follows.

\begin{lem} \label{lem: countexcessive}
Suppose that $L$ is a $k$-assignment for the graph $G$, and suppose that $f$ is a proportional $L$-coloring of $G$.
Then the number of almost excessive colors with respect to $f$ is
$$ \frac{1}{k} \sum_{l \in \mathcal{L}} r_l.$$
\end{lem}

\begin{proof}
Assume $|V(G)|=n$.  Let $A$ be the set of well distributed colors in $\mathcal{L}$.  Let $B$ be the set of almost excessive colors and let $D$ be the set of almost deficient colors with respect to $f$.  Since $f$ is a proportional $L$-coloring, every color in $\mathcal{L}$ is either perfectly used, almost excessive, or almost deficient with respect to $f$.  We calculate that:
\begin{align*}
n = \sum_{l \in \mathcal{L}} |f^{-1}(l)| &= \sum_{a \in A} |f^{-1}(a)| + \sum_{b \in B} |f^{-1}(b)| + \sum_{d \in D} |f^{-1}(d)| \\
&= \sum_{a \in A} \frac{\eta(a)}{k} + \sum_{b \in B} \frac{\eta(b) + k - r_b}{k} + \sum_{d \in D} \frac{\eta(d) - r_d}{k} \\
&= |B| + \frac{1}{k} \left ( \sum_{a \in A} (\eta(a) - r_a) + \sum_{b \in B} (\eta(b) - r_b) + \sum_{d \in D} (\eta(d) - r_d)  \right ) \\
&= |B| + \frac{1}{k} \left ( \sum_{l \in \mathcal{L}} \eta(l) - \sum_{l \in \mathcal{L}} r_l  \right ) \\
&= |B| + \frac{1}{k} \left ( nk - \sum_{l \in \mathcal{L}} r_l  \right ) \\
&= |B| + n - \frac{1}{k} \sum_{l \in \mathcal{L}} r_l.
\end{align*}
From this calculation it is easy to deduce that $|B| = \frac{1}{k} \sum_{l \in \mathcal{L}} r_l.$
\end{proof}

\subsection{Matchings and Monotonicity}  We now show that proportional choosability is monotonic in $k$.  The proof relies on some matching theory.  So, we begin with a quick review of the necessary matching theory, and we introduce some notation.  This notation will also be used when we obtain some further results via matching theory in Section~\ref{matching}.

Using matching theory to prove results about list coloring is not new (see~\cite{ET79}).  A \emph{matching} in a graph $G$ is a set of edges with no shared endpoints.  If $M$ is a matching in $G$, and $X \subseteq V(G)$ such that each vertex in $X$ is an endpoint of an edge in $M$, we say that $X$ is \emph{saturated} by $M$.  If $M$ is a matching in $G$ that saturates $V(G)$, then we say that $M$ is a \emph{perfect matching} of $G$ and \emph{$G$ has a perfect matching}.  The following classic result was proven by Hall in 1935.

\begin{thm}[\cite{H35}, {\bf Hall's Theorem}] \label{thm: Halls}
Suppose $B$ is a bipartite multigraph with bipartition $X, Y$.  Then, $B$ has a matching that saturates $X$ if and only if $|N_B(S)| \geq |S|$ for all $S \subseteq X$.
\end{thm}

It is easy to prove the following corollary from Hall's Theorem (see~\cite{W01}).

\begin{cor} \label{cor: matchdecomp}
If $B$ is a $k$-regular bipartite multigraph, then $E(B)$ can be partitioned into $k$ perfect matchings in $B$.
\end{cor}

We will introduce a specially constructed auxiliary $k$-regular bipartite multigraph, $B$, where one partite set corresponds to colors in $\mathcal{L}$ and the other corresponds to vertices in $G$.  We then use a perfect matching in $B$ to obtain a proportional $L$-labelling, and under additional constraints we can ensure that this labelling is actually a proper coloring.

We now go through the specifics.  For the following lemmas, it is convenient to relax some previous notions.  A \emph{$k$-multi-assignment} assigns to each vertex a multiset of $k$ colors (allowing repeats).  In this context, the multiplicity of a color is the sum of its multiplicities across all lists, and we may still ask for a proportional coloring from such an assignment.

\begin{lem}\label{lem:allmultk}
Let $L$ be a $k$-multi-assignment for $G$ such that $\eta(c)=k$ for every $c\in \mathcal{L}$.  Then there is a proportional $L$-coloring of $G$.  Furthermore, for every $v_0\in V(G)$ and every $c_0\in L(v_0)$, there exists a proportional $L$-coloring $f$ such that $f(v_0)=c_0$.
\end{lem}
\begin{proof}
Define an auxiliary bipartite multigraph with partite sets $V(G)$ and $\mathcal{L}$, and create edges so that there are $\alpha$ edges with endpoints $v$ and $c$ if and only if $c$ appears in $L(v)$ $\alpha$ times.  By hypothesis, this graph is $k$-regular.  By Corollary~\ref{cor: matchdecomp}, there is a perfect matching containing an edge with endpoints $v_0$ and $c_0$.  A perfect matching in this bigraph defines an $L$-coloring of $G$; since each color is used exactly once, it is a proportional $L$-coloring.
\end{proof}

Given a graph $G$ with a $k$-assignment $L$, we define the \emph{well-distributed expansion} of $(G,L)$ as a pair $(G',L')$ constructed as follows.  Let $G'$ be $G$ plus a new isolated vertex $v_c$ for each color $c$ that is not well distributed in $L$.  For each such $v_c$, let $L'(v_c)$ be a multiset consisting of $k-r_{c}$ copies of color $c$ and $r_c$ copies of a new color $c^* \notin \mathcal{L}$.  For $v\in V(G)$, let $L'(v)=L(v)$.  Note that $\eta(c^*)=\sum_{v_c}r_c$, which is clearly divisible by $k$.  Thus, $L'$ is a $k$-multi-assignment of $V(G')$ for which every color is well distributed.

Given a (multi-)list assignment $L$, we call a function $f: V(G)\to \mathcal{L}$ with $f(v)\in L(v)$ for every $v\in V(G)$ (that need not be a proper coloring) an \emph{$L$-labelling}. An $L$-labelling is \emph{proportional} if for every $c\in\mathcal{L}$, $\lfloor\eta(c)/k\rfloor\leq|f^{-1}(c)|\leq\lceil\eta(c)/k\rceil$.

\begin{lem}\label{lem:wellexp}
If $(G',L')$ is a well-distributed expansion of $(G,L)$ and $f$ is a proportional $L'$-labelling of $G'$, then $f$ restricted to $V(G)$ is a proportional $L$-labelling of $G$.
\end{lem}
\begin{proof}
Let $f'$ be a proportional $L'$-labelling of $G'$, and consider $f=f'|_{V(G)}$.  We have that $f$ is an $L$-labelling of $G$.  Each color that is well distributed in $L$ does not appear in any new lists of $L'$, so that color is used perfectly by $f$. Each color $c$ that is not well distributed in $L$ is used $q_c+1$ times by $f'$; if it is used on $v_c$, then it is used almost deficiently by $f$, and if it is not used on $v_c$, then it is used almost excessively by $f$.
\end{proof}

Given a $k$-assignment $L$ for $G$, a \emph{huing} of $L$ is a new $k$-assignment obtained by replacing each $c\in\mathcal{L}$ by $\lceil\eta(c)/k\rceil$ new colors called the \emph{hues of $c$}, with each hue occurring in exactly $k$ lists except perhaps the last (if $r_c>0$) which is called the \emph{scarce hue of $c$}, which appears in $r_c$ lists.
We say that a huing is \emph{good} if $G$ has no edges joining vertices whose lists contain different hues of the same color.


\begin{lem}\label{lem:hues}
Let $L$ be a $k$-assignment for $G$, let $\widetilde{L}$ be any huing of $(G,L)$, and pick any $v_0\in V(G)$ and $c_0\in L(v_0)$.
Then there is a proportional $L$-labelling of $G$, $f$, with $f(v_0)=c_0$, and if $\widetilde{L}$ is a good huing then $f$ is a proportional
$L$-coloring of $G$.
\end{lem}
\begin{proof}
Let $(G',L')$ be a well-distributed expansion of $(G,L)$.
The huing $\widetilde{L}$ extends naturally to a huing $\widetilde{L'}$ of $L'$: the colors in lists of $G$ are hued according to $\widetilde{L}$, the appearances of color $c$ in $L'(v_c)$ are changed to the scarce hue of $c$, and the additional color $c^*$ added in the expansion is hued arbitrarily.

Lemma~\ref{lem:allmultk} gives us a proportional $\widetilde{L'}$-coloring of $G'$ such that the color on $v_0$ is a hue of $c_0$.  Restricting it to $V(G)$ yields a proportional $\widetilde{L}$-labelling $\widetilde{f}$ of $G$ by Lemma~\ref{lem:wellexp}.
Projecting hues back to their original colors preserves the proportional use of each color, so that results in a proportional $L$-labelling $f$.

If $uv\in E(G)$ and $f(u)=f(v)=c$, then $\widetilde{f}(u)\neq\widetilde{f}(v)$ since $\widetilde{f}$ was inherited from a $\widetilde{L'}$-coloring of $G'$, so $\widetilde{f}(u)$ and $\widetilde{f}(v)$ must be two different hues of $c$.  That cannot happen if $\widetilde{L}$ is a good huing.
\end{proof}

Finally, we are ready to prove monotonicity in $k$.
\begin{pro}\label{pro: monoink}
If $G$ is proportionally $k$-choosable, then $G$ is proportionally $(k+1)$-choosable.
\end{pro}
\begin{proof}
Let $L$ be a $(k+1)$-assignment for $G$.  Let $\widetilde{L}$ be any huing; Lemma~\ref{lem:hues} gives a proportional $L$-labelling $\varphi$ of $G$.
Define the $k$-assignment $L'$ by $L'(v)=L(v)\setminus\{\varphi(v)\}$.
By the hypothesis, $G$ has a proportional $L'$-coloring~$f$. Let $c\in\mathcal{L}$.

If $c$ is well distributed in $L$, then $\eta_{L'}(c)= k\cdot\frac{\eta_L(c)}{k+1}$, so that $c$ is well distributed in $L'$ and $|f^{-1}(c)|=\frac{\eta_{L'}(c)}{k}=\frac{\eta_L(c)}{k+1}$ as required.

If $c$ is not well distributed in $L$, say $\eta_L(c)=(k+1)q_c+r_c$ with $0<r_c\leq k$, then $c$ is used $q_c+\delta_c$ times by $\varphi$, where $\delta_c=0$ if $c$ is used almost deficiently with respect to $\varphi$ and $\delta_c=1$ if $c$ is used almost excessively with respect to $\varphi$.  So
\[ \eta_{L'}(c) = ((k+1)q_c+r_c)-(q_c+\delta_c) = kq_c+r_c-\delta_c. \]

This gives $|f^{-1}(c)|\geq \floor{\frac{\eta_{L'}(c)}{k}} \geq q_c$, and $|f^{-1}(c)|\leq \ceil{\frac{\eta_{L'}(c)}{k}} \leq q_c+1$.  Thus $f$ is a proportional $L$-coloring of $G$.
\end{proof}

Having proven monotonicity in $k$, we are ready for a definition.  For any graph $G$, the \emph{proportional choice number} of $G$, denoted $\chi_{pc}(G)$, is the smallest $k$ such that $G$ is proportionally $k$-choosable.  By Proposition~\ref{pro: monoink}, we know that any graph $G$ is proportionally $k$-choosable if and only if $k \geq \chi_{pc}(G)$.

The next result shows that the property of proportional $k$-choosability is monotone in the subgraph relation.  This is another fact of which we will make frequent use.

\begin{pro} \label{lem: monotone}
Suppose $H$ is a subgraph of $G$.  If $G$ is proportionally $k$-choosable, then $H$ is proportionally $k$-choosable.
\end{pro}

\begin{proof}
Let $L'$ be an arbitrary $k$-assignment for $H$.  Let $\mathcal{L}' = \bigcup_{v \in V(H)} L'(v)$, and assume $\mathcal{L}' \subset \N$.  Let $M = \max \mathcal{L}'$.  Now, let $L$ be the $k$-assignment for $G$ given by:
\[
  L(v) =
  \begin{cases}
                                   \{M+1, M+2, \ldots, M+k \} & \text{if} \; v \in V(G)-V(H) \\
                                   L'(v) & \text{if} \; v \in V(H)
  \end{cases}
\]
Since $G$ is proportionally $k$-choosable, there is a proportional $L$-coloring, $f$, of $G$.  We note that if we restrict $f$ to $V(H)$ we obtain a proportional $L'$-coloring of $H$.
\end{proof}

\subsection{First Upper Bounds}  We begin with an easy upper bound on the proportional choice number of a graph.

\begin{pro} \label{pro: order}
For any $G$, $\chi_{pc}(G) \leq |V(G)|$.
\end{pro}

\begin{proof}
We need to show that $G$ is proportionally $|V(G)|$-choosable.  Fix a $|V(G)|$-assignment, $L$, for $G$.  For each $c \in \mathcal{L}$, $\eta(c)/|V(G)| \leq |V(G)|/|V(G)| = 1$.  So, we must pick each color at most once.  Furthermore, we must use $c \in \mathcal{L}$ exactly once if and only if $\eta(c)=|V(G)|$.  Let $A$ be the set of colors that must be used exactly once.  Assign each color in $A$ to one of $|A|$ distinct vertices (pick any subset of $V(G)$ of size $|A|$) in a one-to-one fashion, and remove those colors from the lists of the other vertices.  It remains to color $|V(G)|-|A|$ vertices with $|V(G)|-|A|$ distinct colors from lists of size $|V(G)|-|A|$, which we can do greedily.
\end{proof}

It is easy to see that the bound in Proposition~\ref{pro: order} is tight since $\chi_{pc}(K_n)=n$.  With a bit more effort, we can slightly extend Proposition~\ref{pro: order} by showing that $K_n$ is the only graph achieving equality there.

\begin{pro} \label{pro: orderplus1}
Suppose that $G$ is a graph on $k+1$ vertices such that $G \neq K_{k+1}$.  Then, $G$ is proportionally $k$-choosable. That is, $\chi_{pc}(G) \leq |V(G)|-1$ whenever $G$ is not a complete graph.
\end{pro}

\begin{proof}
Note that the result is obvious when $k=0,1$.  So, we assume that $k \geq 2$ throughout this proof.  Suppose that $L$ is an arbitrary $k$-assignment for $G$.  We must show that there is a proportional $L$-coloring for $G$.  First, note that a proportional $L$-coloring clearly exists when $L$ is a constant $k$-assignment (i.e. a list assignment where all the lists are the same) since in this case we can color 2 nonadjacent vertices of $G$ with the same color and then greedily color what remains to obtain a proportional $L$-coloring.

\par

So, we assume that $L$ is a non-constant $k$-assignment for $G$.  We let
$$A = \{ c \in \mathcal{L} : \eta(c) \geq k \}.$$
Note that a proportional $L$-coloring of $G$ must use each color in $A$ at least one time.  Moreover, for any color in $\mathcal{L} - A$, a proportional $L$-coloring must use the color 0 or 1 time.  Suppose $A = \{c_1, \ldots, c_m \}$.  Clearly $0 \leq m \leq k+1$.  We will prove the desired in the case that (1) $m \leq k$, and in the case that (2) $m= k + 1$.

\par

For case (1), we find a vertex, $v_1 \in V(G)$, such that $c_1 \in L(v_1)$, and we color $v_1$ with $c_1$.  Then, we proceed inductively, and for each $j$ satisfying $2 \leq j \leq m$, we find a vertex, $v_j \in V(G) - \{v_1, \ldots, v_{j-1} \}$, such that $c_j \in L(v_j)$ (such a vertex must exist since $c_j$ appears in the list corresponding to at least $k$ vertices in $V(G)$), and we color $v_j$ with $c_j$.  We let $S= \{v_1, \ldots, v_m \}$ (note that $S$ will be empty if $m=0$), and we let $S'=V(G)-S$.  Now, for each vertex $v \in S'$, let $L'(v) = L(v)-A$.  Note that each color in $\bigcup_{v \in S'} L'(v)$ has multiplicity less than $k$ in $L$.  So, we can complete a proportional $L$-coloring of $G$ if we can find a proper $L'$-coloring of $G[S']$ that uses $|S'|=k-m+1$ pairwise distinct colors.

\par

We note that $|L'(v)| \geq k-m$ for each $v \in S'$.  We can find a proper $L'$-coloring of $G[S']$ that uses $|S'|=k-m+1$ pairwise distinct colors if $L'$ is not a constant $(k-m)$-assignment for $G[S']$.  So, we assume $L'$ is a constant $(k-m)$-assignment for $G[S']$ (note that since $L$ is not a constant $k$-assignment, it must be that $m \geq 1$).  Specifically, suppose that $L'(v) = B$ for each $v \in S'$.  This means that $L(v) = B \cup A$ for each $v \in S'$.  Since $L$ is not a constant $k$-assignment, there must be some $v_i \in S$ such that $L(v_i) \neq B \cup A$.  Now, pick any vertex $u \in S'$ and color it with $c_i$.  Then, make it so that $v_i$ is uncolored.  Now, let
$$S'' = (S' - \{u \}) \cup \{v_i \}$$
and let $L''(v) = L(v)-A$ for each $v \in S''$.  Note that each color in $\bigcup_{v \in S''} L''(v)$ has multiplicity less than $k$ in $L$.  Since $L(v_i) \neq B \cup A$, we have that $L''$ is not a constant $(k-m)$-assignment for $G[S'']$.  Thus, we can find a proper $L''$-coloring of $G[S'']$ that uses $|S''|=k-m+1$ pairwise distinct colors which completes a proportional $L$-coloring of $G$.  This completes case (1).

\par

For case (2) we suppose that $m = k+1$.  In this case note that since each color in $A$ has multiplicity at least $k$ in $L$,
$$ k(k+1) \leq \sum_{i=1}^{k+1} \eta(c_i).$$
Also, since $L$ assigns to each vertex in $G$ a list of $k$ colors, we know that
$$ k(k+1) \leq \sum_{i=1}^{k+1} \eta(c_i) \leq \sum_{c \in \mathcal{L}} \eta(c) = k|V(G)| = k(k+1). $$
So, it must be that each color in $A$ has multiplicity $k$ and $A = \mathcal{L}$.  So, we may assume that $\mathcal{L} = \{1, 2, \ldots, k+1 \}$, and $L(v)$ is a $k$-element subset of $\mathcal{L}$ for each $v \in V(G)$.  Moreover, since each color in $A$ does not appear in exactly one list obtained from $L$, we know that each $k$-element subset of $\mathcal{L}$ is assigned to exactly one vertex in $V(G)$ by $L$.  We let $v_i \in V(G)$ be the vertex with the property $i \notin L(v_i)$.  To obtain a proportional $L$-coloring of $G$, we simply color $v_i$ with $i+1$ for each $1 \leq i \leq k$, and we color $v_{k+1}$ with 1. This completes case (2), and we are finished.
\end{proof}

\subsection{Lower Bounds}

We now make some progress on Theorem~\ref{thm: fullstars} by showing that stars with too many leaves fail to be proportionally $k$-choosable.

\begin{pro} \label{pro: starsec}
$K_{1,2k-1}$ is not proportionally $k$-choosable for each $k \in \N$.
\end{pro}

\begin{proof}
Let $G = K_{1,2k-1}$.  Note that in any proper $k$-coloring of $G$, the color used on the vertex in the partite set of size one must be used exactly once.  However, $\lfloor |V(G)|/k \rfloor = 2$.  So, $G$ is not equitably $k$-colorable, and Proposition~\ref{pro: motivation} implies $G$ is not proportionally $k$-choosable.
\end{proof}

Propositions~\ref{lem: monotone} and~\ref{pro: starsec} immediately yield the following which is a negative result in the spirit of Question~\ref{ques: fullHS}.

\begin{cor} \label{cor: maxdegreesec}
Suppose that $G$ is a graph with $\Delta(G) \geq 2k-1$ for $k \in \N$.  Then, $G$ is not proportionally $k$-choosable.  That is,
$$\chi_{pc}(G) > \frac{\Delta(G)+1}{2}$$
for any graph $G$.
\end{cor}

We now present the result that lead to the formulation of Question~\ref{ques: proOhba}.  One should notice that the result below is not implied by Corollary~\ref{cor: maxdegreesec} since $(\Delta(K_{2*m}) +1)/2 = m - \frac{1}{2}$.

\begin{pro} \label{pro: complete2}
If $m \geq 2$, then $\chi_{pc}(K_{2*m}) > m$.
\end{pro}

\begin{proof}
Suppose that $G$ is a copy of $K_{2*m}$ with partite sets $A_1, A_2, \ldots, A_m$.  For $1 \leq i \leq m$, suppose $A_i = \{a_{i,1}, a_{i,2} \}$.  Now, suppose that $L$ is the $m$-assignment for $G$ obtained by letting $L(a_{i,l}) = \{1, 2, \ldots, m-1 \} \cup \{m-1+i \}$ for each $i=1,2, \ldots, m$ and $l=1,2$.  Notice that $\mathcal{L} = \{1, 2, \ldots, 2m-1 \}$.  Also, $\eta(j) = 2m$ when $1 \leq j \leq m-1$, and $\eta(j)=2$ when $m \leq j \leq 2m-1$.

\par

For the sake of contradiction, suppose that $f$ is a proportional $L$-coloring of $G$.  We have that $f$ must use each of the colors: $1, 2, \ldots, m-1$ exactly 2 times, and $f$ must use all other colors in $\mathcal{L}$ at most one time.  Since $\alpha(G)=2$, and there are exactly $m$ independent sets in $G$ of size 2 (namely $A_1, A_2, \ldots, A_m$), we may assume without loss of generality that $f^{-1}(j) = A_j$ for $j=1, 2, \ldots, m-1$.  Then, in order for $f$ to be a proper coloring, $f$ must color both of the vertices in $A_m$ with the color $2m-1$.  This however contradicts the fact that $f$ uses $2m-1$ at most one time.  So, there is no proportional $L$-coloring of $G$, and we know that $G$ is not proportionally $m$-choosable.
\end{proof}

We now present two results which show that certain forests and complete bipartite graphs of max degree $\Delta$ need not be proportionally $\Delta$-choosable.

\begin{pro} \label{pro: maxdegreesec}
Suppose $H_1, H_2, \ldots, H_k$ are $k$ pairwise vertex disjoint copies of $K_{1,k}$.  If $G = \sum_{i=1}^k H_i$, then $G$ is not proportionally $k$-choosable.
\end{pro}

\begin{proof}
The result is obvious when $k=1$.  So, we may assume that $k \geq 2$ throughout the proof.  For $1 \leq i \leq k$ suppose that $H_i$ has bipartition $A_i$, $B_i$ where $A_i = \{a_i \}$ and $B_i = \{b_{k(i-1)}, b_{k(i-1)+1}, \ldots, b_{ki-1} \}$.  We will now construct a $k$-assignment, $L$, for $G$ such that there is no proportional $L$-coloring for $G$.

\par

To begin, we let $L(a_i) = \{1, 2, \ldots, k \}$ for $1 \leq i \leq k$.  Then, for $1 \leq i \leq k$ and $0 \leq j \leq k-1$, we let
$$L(b_{k(i-1)+j}) = \{1, i(k-1)+2, i(k-1)+3, \ldots, (i+1)(k-1)+1 \}.$$
It is easy to see that
$$ \bigcup_{v \in V(G)} L(v) = \{1, 2, \ldots, k^2 \}.$$
Moreover, $\eta(1) = |V(G)| = k^2+k$ and $\eta(j) = k$ for each each $2 \leq j \leq k^2$.  Now, for the sake of contradiction, suppose that $f$ is a proportional $L$-coloring of $G$.

\par

We first claim that $f(a_i) \neq 1$ for each $1 \leq i \leq k$.  To see why, suppose that $f(a_l)=1$ for some $l$.  Since $f$ is a proper $L$-coloring, this implies that $f(B_l) \subseteq \{l(k-1)+2, \ldots, (l+1)(k-1)+1 \}$ which immediately implies that
$$|f(B_l)| \leq k-1.$$
Since $|B_l|=k$, the pigeonhole principle implies that there exists two vertices in $B_l$ that are given the same color, say $t$, by $f$.  Since $t > 1$, $t$ is used too many times by $f$ which is a contradiction.  Thus, $f(a_i) \neq 1$ for each $1 \leq i \leq k$.

\par

Now, we know that $f(\{a_1, \ldots, a_k \}) \subseteq \{2, 3, \ldots, k \}$ which immediately implies that
$$|f(\{a_1, \ldots, a_k \})| \leq k-1.$$
Since $|\{a_1, \ldots, a_k \}|=k$, the pigeonhole principle implies that there exist two vertices in $\{a_1, \ldots, a_k \}$ that are given the same color, say $t$, by $f$.  Since $t > 1$, $t$ is used too many times by $f$ which is a contradiction.  Thus, no proportional $L$-coloring of $G$ exists and our proof is complete.
\end{proof}

Since Proposition~\ref{pro: orderplus1} implies $K_{1,k}$ is proportionally $k$-choosable, Proposition~\ref{pro: maxdegreesec} also shows that, like equitable choosability, adding more components to a graph may make finding a proportional coloring more difficult.  Thus, techniques we use to prove results about proportional choosability of connected graphs may not work in the context of disconnected graphs.

\begin{pro} \label{pro: completebipartitesec}
$K_{m,m}$ is not proportionally $m$-choosable for each $m \in \N$.
\end{pro}

\begin{proof}
First, note that the result is obvious when $m=1$.  So, we assume that $m \geq 2$, and let $G=K_{m,m}$.  Suppose that $G$ has bipartition $A$, $B$ where $A = \{a_1, \ldots, a_m \}$ and $B= \{b_1, \ldots, b_m \}.$   We will now construct an $m$-assignment, $L$, for $G$ such that there is no proportional $L$-coloring for $G$.  For $1 \leq i \leq m$, let $L(a_i) = \{1, 2, \ldots, m \}$ and $L(b_i) = \{1, m+1, m+2, \ldots, 2m-1 \}$.  Note that $\eta(1) = 2m$ and $\eta(j) = m$ for $2 \leq j \leq 2m-1$.

\par

For the sake of contradiction, suppose that $f$ is a proportional $L$-coloring of $G$.  Since $|f^{-1}(1)|=2$, we may assume without loss of generality that $f$ colors two vertices in $A$ with 1.  This however makes it impossible for $f$ to use each color in $\{2,3, \ldots, m \}$ exactly one time since the colors in $\{2,3, \ldots, m \}$ only appear in lists corresponding to vertices in $A$.  This is a contradiction, and we conclude there is no proportional $L$-coloring for $G$.
\end{proof}

Note that Proposition~\ref{pro: completebipartitesec} implies that the proportional analogue of Conjecture~\ref{conj: KPW2} does not hold.  Moreover, Propositions~\ref{lem: monotone} and~\ref{pro: maxdegreesec} make it easy to construct trees where the proportional analogue of Conjecture~\ref{conj: KPW2} does not hold.

\section{Proportional Choosability of Small Graphs} \label{smallorder}

In this section we prove Theorem~\ref{thm: smallorder}: $\chi_{pc}(G) \leq \Delta(G) + \lceil |V(G)|/2 \rceil$ for every graph $G$.  In order to prove Theorem~\ref{thm: smallorder}, we will prove two lemmas which will make the proof of Theorem~\ref{thm: smallorder} trivial.  We begin with a lemma that is interesting in its own right.

\begin{lem} \label{lem: algorithm}
Suppose $G$ is a graph, and suppose that $L$ is a $k$-assignment for $G$ such that
$$\max_{c \in \mathcal{L}} \eta(c) < 2k.$$
If there exists a proper $L$-coloring of $G$ that uses at most $t$ colors excessively, then there exists a proper $L$-coloring of $G$, $g$, such that no color in $\mathcal{L}$ is deficient with respect to $g$ and at most $t$ colors in $\mathcal{L}$ are excessive with respect to $g$.  In particular, if  there is a proper $L$-coloring of $G$ that uses no color $c \in \mathcal{L}$ excessively, then $G$ is proportionally $L$-colorable.
\end{lem}

\begin{proof}
Among all the proper $L$-colorings of $G$ that use at most $t$ colors in $\mathcal{L}$ excessively, let $f$ be such that the number of deficient colors with respect to $f$ is as small as possible.  For the sake of contradiction, suppose there is at least one deficient color with respect to $f$.

\par

Let $X = \{c \in \mathcal{L} : |f^{-1}(c)| < \lfloor \eta(c)/k \rfloor \}$, $Y = \{c \in \mathcal{L} : |f^{-1}(c)| = \lfloor \eta(c)/k \rfloor \}$, and $Z = \{c \in \mathcal{L} : |f^{-1}(c)| > \lfloor \eta(c)/k \rfloor \}$.  Note that for each $c \in \mathcal{L}$, $\lfloor \eta(c)/k \rfloor$ is 0 or 1 since $\eta(c) < 2k$.  So, the colors in $X$ have multiplicity at least $k$, and they are not used by $f$.  Therefore, $X$ is precisely the colors in $\mathcal{L}$ which are deficient with respect to $f$, and by our assumption, $|X| \geq 1$.

\par

We define an auxiliary digraph $H$ with $V(H)=V(G)$ as follows.  For each vertex $u \in V(H)$, let $(u,v) \in E(H)$ for each $v$ with $f(u) \in L(v)$ (note that we allow loops in this graph).  If $P$ is a directed path in $H$ with vertices (written in order): $v_1, v_2, \ldots, v_m$, by \emph{shifting colors along $P$} we mean that we modify the coloring $f$ as follows: we remove the color $f(v_1)$ from $v_1$ and for $2 \leq i \leq m$ we recolor $v_i$ with $f(v_{i-1})$.  By the way $H$ is constructed, the coloring we obtain after shifting colors along $P$ assigns to each vertex in $V(G) - \{v_1 \}$ a color in its corresponding list.

\par

Let $S_0$ be the set of vertices in $V(G)$ whose lists contain at least one deficient color with respect to $f$ (note that $|S_0| \geq k$).  Let $S$ be the set of vertices reachable from $S_0$ via a directed path in $H$ (note that we allow directed paths of length 0).  We claim that $f(S)$ must contain an element in $Z$.

\par

To see why, suppose for the sake of contradiction that $f(S) \subseteq Y$.  Let $H' = H[S]$.  By the way $S$ is defined, $d^+_{H'}(v) = d^+_{H}(v)$ for each $v \in S$.  For each $v \in S$, we know that $\eta(f(v)) \geq k$ which means that $k \leq d^+_{H}(v) = d^+_{H'}(v)$.  Also, for each $v \in S$,
$$d^-_{H'}(v) = | \{ u \in S : f(u) \in L(v) \} |.$$
Since each color in $f(S)$ is used exactly once by $f$ and $L$ is a $k$-assignment, $d^-_{H'}(v) \leq k$ for each $v \in S$.  Furthermore, since each vertex in $S_0$ contains at least one color never used by $f$ in its corresponding list, $d^-_{H'}(v) < k$ for each $v \in S_0$.  Thus, we have
$$k|S| \leq \sum_{v \in S} d^+_{H'}(v) = \sum_{v \in S} d^-_{H'}(v) = \sum_{v \in S_0} d^-_{H'}(v) + \sum_{v \in S-S_0} d^-_{H'}(v) < k(|S_0| + |S-S_0|) = k|S|$$
which is a contradiction.  We conclude that some element of $f(S)$ is in $Z$.

\par

Let $Z'$ be all the vertices in $S$ that are colored by $f$ with a color in $Z$.  Suppose that $P$ is a directed path (in $H$) of minimum length between $S_0$ and $Z'$.  Suppose that the vertices of $P$ (in order) are: $y_1, y_2, \ldots, y_m$.  Clearly, $y_1 \in S_0$, $f(y_m) \in Z$, and since $P$ is of minimum length, $f(y_i) \in Y$ for $i=1, \ldots, m-1$.  Suppose that $c$ is a deficient color in $L(y_1)$.  Let $f'$ be the coloring for $G$ obtained from $f$ by shifting colors along $P$, and coloring $y_1$ with $c$.  Since $c$ was never used by $f$, and each of: $f(y_1), \ldots, f(y_{m-1})$ is used exactly one time by $f$, $f'$ is a proper $L$-coloring of $G$, and the only color classes associated with $f$ and $f'$ that differ in size are the color classes corresponding to $c$ and $f(y_m)$.  Specifically, $f'$ uses $c$ one more time than $f$, and it uses $f(y_m)$ one less time than $f$.  So, $f(y_m)$ is not deficient with respect to $f'$, and $c$ is almost deficient with respect to $f'$.  We conclude that $f'$ is a proper $L$-coloring of $G$ that uses at most $t$ colors excessively, and there are fewer colors in $\mathcal{L}$ deficient with respect to $f'$ than with respect to $f$.  This is a contradiction.
\end{proof}

Lemma~\ref{lem: algorithm} gives us a way to remove deficient colors when our graph has small order.  Our next lemma will tell us that we can also avoid excessive colors when our graph satisfies further conditions.  Both these lemmas will allow us to easily prove Theorem~\ref{thm: smallorder}.

\begin{lem} \label{lem: halfsec}
Suppose $G$ is a graph, $l \geq 2$, and $k \geq \max\{l \Delta(G) , l|V(G)|/(2(l-1)) \}$.  If $L$ is a $k$-assignment for $G$, then there exists a proper $L$-coloring of $G$ that uses no color $c \in \mathcal{L}$ more than $\lceil \eta(c)/k \rceil$ times (i.e. excessively).
\end{lem}

\begin{proof}

Suppose for the sake of contradiction that the lemma is false.  Let $G$ be a graph that violates the lemma with the smallest possible number of vertices.  We let $|V(G)|=n$.  Note that if $n$ is 1 or 2, $ \lceil \max\{l \Delta(G) , l|V(G)|/(2(l-1)) \} \rceil \geq n \geq \chi_{pc}(G)$ which means the lemma holds when $n=1,2$.  So, $n \geq 3$.  Let $L$ be a $k$-assignment with $k$ and $l$ satisfying the condition such that a proper $L$-coloring of $G$ satisfying the desired property does not exist.  We have that $n \leq 2k(1-1/l)$ and $\Delta(G) \leq k/l$.

\par

Suppose that $y \in V(G)$, and let $G' = G-y$ and $K = L|_{V(G')}$.  By the minimality of $G$ and the fact $\Delta(G') \leq \Delta(G)$ and $|V(G')|<|V(G)|$, there is a proper $K$-coloring of $G'$, $f$, such that for each $c \in \mathcal{L}$,
$$|f^{-1}(c)| \leq \left \lceil \frac{\eta_{K}(c)}{k} \right \rceil \leq \left \lceil \frac{\eta_{L}(c)}{k} \right \rceil \leq \left \lceil \frac{n}{k} \right \rceil \leq \left \lceil 2 - \frac{2}{l} \right \rceil \leq 2.$$
Now, for $i=1,2$, let $E_i$ denote the elements of $c \in \mathcal{L}$ that satisfy $|f^{-1}(c)| = i = \lceil \eta_L(c)/k \rceil$, and for $j=0,1$, let $D_j$ denote the elements of $c \in \mathcal{L}$ that satisfy $|f^{-1}(c)| = j < \lceil \eta_L(c)/k \rceil$.  It is easy to see that $\{E_1, E_2, D_0, D_1\}$ is a partition of $\mathcal{L}$.

\par

We let $L'(y) = L(y) - \{f(v) : v \in N_G(y) \}$.  Clearly $|L'(y)| \geq k - \Delta(G) \geq k(1-1/l)$.  Note that it must be the case that $L'(y) \cap (D_1 \cup D_0) = \emptyset$.  Otherwise we could simply color $y$ with some $c \in L'(y) \cap (D_1 \cup D_0)$ to complete a coloring with the desired property.  So, $L'(y) \subseteq E_1 \cup E_2$.  Let $U = \{v \in V(G') : f(v) \in L'(y) \}$.  Since each color in $L'(y)$ is used at least once by $f$, we have that $|U| \geq |L'(y)|$.  We will now derive a contradiction in two possible cases: (1) $D_1 \neq \emptyset$ and (2) $D_1 = \emptyset$.

\par

For case (1) suppose that $c^* \in D_1$ and $f(v^*) = c^*$.  Since $c^* \in D_1$, it follows that $\eta_L(c^*) \geq k+1$.  Now, let $R = \{v \in V(G) : c^* \in L(v), v \notin N_G[v^*] \}$.  We note that $R$ is a subset of $V(G')$ since $c^* \notin L'(y)$ implies $y \notin R$. Also, $|R| \geq k + 1 - 1 - \Delta(G) \geq k(1-1/l)$.  Now, we claim that $U$ and $R$ are not disjoint.  To see why this is so, suppose that $U$ and $R$ are disjoint subsets of $V(G')$.  Then
$$2k \left (1-\frac{1}{l} \right ) - 1 \geq n - 1 = |V(G')| \geq |U| + |R| \geq 2k \left (1-\frac{1}{l} \right )$$
which is a contradiction.  Suppose that $x \in V \cap R$.  Now, starting from $f$, recolor $x$ with $c^*$ and color $y$ with $f(x)$.  The resulting coloring is a proper $L$-coloring of $G$ with the desired property.  This is a contradiction which completes case (1).

\par

For case (2) we know that $D_1 = \emptyset$.  We claim that $D_0$ contains at least one element.  To see why this is so, note that:
$$2|E_2|+|E_1| = n - 1.$$
Also,
\begin{align*}
(2|E_2|+|E_1|+1)k = nk &= \sum_{e_2 \in E_2} \eta_L(e_2) + \sum_{e_1 \in E_1} \eta_L(e_1) + \sum_{d_0 \in D_0} \eta_L(d_0) \\
& \leq n |E_2| + k |E_1| + \sum_{d_0 \in D_0} \eta_L(d_0) \\
& \leq \frac{2k(l-1)}{l} |E_2| + k |E_1| + \sum_{d_0 \in D_0} \eta_L(d_0).
\end{align*}
This implies that:
$$2|E_2|+|E_1|+1 \leq \frac{2(l-1)}{l}|E_2| + |E_1| + \frac{1}{k}\sum_{d_0 \in D_0} \eta_L(d_0) \implies \frac{2}{l} |E_2| + 1 \leq \frac{1}{k}\sum_{d_0 \in D_0} \eta_L(d_0).$$
The last inequality immediately implies that $|D_0| \geq 1$.  Similar to the proof of Lemma~\ref{lem: algorithm}, we define an auxiliary digraph $H$ with $V(H)=V(G)$ as follows.  For each vertex $u \in V(H)$, let $(u,v) \in E(H)$ for each $v$ with $f(u) \in L(v)$.  Note that since $y$ is not in the domain of $f$, $d^+_H(y)=0$.

\par

Let $S_0$ be the set of vertices whose list contains some element of $D_0$, and let $S$ be the set of vertices reachable from $S_0$ by a directed path in $H$ (again we allow directed paths of length 0).  We claim that $y \in S$.  To see why this is so, suppose that $y \notin S$.  By the definition of $S$, each $c \in f(S)$ only appears in lists corresponding to vertices in $S$.  Let $G'' = G - S$, and $K' = L|_{V(G'')}$.  By the minimality of $G$, there is a proper $K'$-coloring, $f'$, of $G''$ such that $|f'^{-1}(c)| \leq \lceil \eta_{K'}(c)/k \rceil \leq \lceil \eta_{L}(c)/k \rceil$ for each $c \in \mathcal{L}$.  Moreover, $f'$ never uses any color in $f(S)$ since the multiplicity of each $c \in f(S)$ is 0 in $K'$.  Let $g$ be the proper $L$-coloring for $G$ given by
\[
  g(v) =
  \begin{cases}
                                   f'(v) & \text{if} \; v \in V(G)-S \\
                                   f(v) & \text{if} \; v \in S
  \end{cases}
\]
(note that we need to have $y \notin S$ for $g$ to be a proper $L$-coloring for $G$).  It is easy to see that for each $c \in \mathcal{L}$, $|g^{-1}(c)| \leq \lceil \eta_{L}(c)/k \rceil$ which contradicts the fact that $G$ violates the lemma.  We conclude that $y \in S$.

\par

Suppose that $P$ is a directed path (in $H$) of minimum length between $S_0$ and $E_2 \cup \{y\}$ (such a $P$ exists since $y \in S$).  Suppose that the vertices of $P$ (in order) are: $z_1, z_2, \ldots, z_m$.  Clearly, $z_1 \in S_0$, $z_m=y$ or $f(y_m) \in E_2$, and since $P$ is of minimum length, $f(y_i) \in E_1$ for $i=1, \ldots, m-1$.  Now, modify $f$ by shifting colors along $P$, and color $z_1$ with an element in $L(z_1) \cap D_0$.

\par

In the case $z_m = y$ we get a proper $L$-coloring for $G$ which satisfies the desired property which is a contradiction.  On the other hand, if $f(y_m) \in E_2$ we still have that $f$ is a proper $K$-coloring of $G'$ which satisfies the desired property, and we have that $f(y_m) \in D_1$.  Thus, we may get a contradiction by proceeding as we did in case (1) since we now have that $D_1 \neq \emptyset$.  This completes the proof.
\end{proof}

We now present the proof of Theorem~\ref{thm: smallorder}.

\begin{customthm}{\bf \ref{thm: smallorder}}
Any graph $G$ with $\Delta(G) \geq 1$ satisfies
$$\chi_{pc}(G) \leq \Delta(G) + \left \lceil \frac{|V(G)|}{2} \right \rceil.$$
\end{customthm}

\begin{proof}
Suppose $G$ is a graph with $\Delta(G) \geq 1$ and $k = \Delta(G) + \lceil |V(G)|/2 \rceil$.  Let $L$ be an arbitrary $k$-assignment for $G$.  We must show a proportional $L$-coloring of $G$ exists.  If $|V(G)| \leq \Delta(G) + \lceil |V(G)|/2 \rceil$, Propositions~\ref{pro: monoink} and~\ref{pro: order} yield the desired result.  So, assume that $|V(G)| > \Delta(G) + \lceil |V(G)|/2 \rceil$.  This implies that $|V(G)|/(2\Delta(G)) > 1$.  So, if we let $l = 1 + |V(G)|/(2\Delta(G))$, we see that Lemma~\ref{lem: halfsec} implies that there exists a proper $L$-coloring of $G$ which uses no color $c \in \mathcal{L}$ excessively.  Finally, since
$$\max_{c \in \mathcal{L}} \eta(c) \leq |V(G)| < 2 \Delta(G) + |V(G)| \leq 2k,$$
Lemma~\ref{lem: algorithm} implies that a proportional $L$-coloring of $G$ exists.
\end{proof}

It should be noted that by Proposition~\ref{pro: orderplus1}, the bound on $\chi_{pc}(G)$ in Theorem~\ref{thm: smallorder} can be taken to be $ \min \{|V(G)|-1, \Delta(G) + \lceil |V(G)|/2 \rceil \}$ if we know that $G$ is not a complete graph.

\section{Further Results via Matching Theory} \label{matching}

In this section we use matching theory to prove Theorems~\ref{thm: fullcomponents} and~\ref{thm: fullstars}.  We use the same terminology that was introduced at the beginning of Section~\ref{easyresults}: whenever $L$ is a $k$-assignment for $G$, for each $c\in \mathcal{L}$, write $\eta(c) = kq_c + r_c$ where $0 \leq r_c \leq k-1$.  We also let $V_c = \{v : c \in L(v) \}$.

\subsection{Proportional Choosability of Disconnected Graphs}  In this subsection we prove Theorem~\ref{thm: fullcomponents}.
In this subsection, given a $k$-assignment for $G$, we let $a_i^c$ denote the number of components of $G[V_c]$ that have order $i$.

\begin{lem} \label{lem: isolated}
Suppose $G$ is a graph that has no component of order greater than $k$ and $L$ is a $k$-assignment for $G$.  If
$$a_1^c \geq \sum_{j=2}^{k} (k-j) a_j^c$$
for each $c \in \mathcal{L}$, then there is a proportional $L$-coloring of $G$.
\end{lem}

\begin{proof}
We hue $L$ as follows.  For each $c \in \mathcal{L}$ and each component $H$ of $G[V_c]$ with $j\geq2$ vertices, we introduce one hue of $c$.  For every vertex in $H$ together with $k-j$ isolated vertices of $V_c$, we replace the occurrence of $c$ by this hue.  The inequality in the hypothesis ensures that we have enough isolated vertices to add to each nontrivial component so that each resulting hue has multiplicity $k$.  Finally, if there are isolated vertices remaining, we hue the appearance of $c$ in their lists arbitrarily (always with multiplicity $k$ except perhaps for one scarce hue).

Notice that by construction, this is a good huing.  So, the desired result follows by Lemma~\ref{lem:hues}.
\end{proof}

We are now ready to prove Theorem~\ref{thm: fullcomponents}.  For the proof, if $G$ has no components of order greater than $k$ and $L$ is a $k$-assignment for $G$, then for each $c \in \mathcal{L}$, we let
$$\delta(L,c) = -a_1^c + \sum_{j=2}^{k} (k-j) a_j^c \; \; \text{and} \; \; \sigma(G,L) = \sum_{c \in \mathcal{L}} \max \{\delta(L,c) , 0 \}.$$

\begin{customthm} {\bf \ref{thm: fullcomponents}}
If $G$ is a graph such that its largest component has $k$ vertices, then $\chi_{pc}(G) \leq k$.
\end{customthm}

\begin{proof}
Suppose the theorem is false.  Among all the graphs with components of order at most $k$ that are not proportionally $k$-choosable, choose a graph $G$ along with a $k$-assignment, $L$, for $G$ such that there is no proportional $L$-coloring of $G$ and $\sigma(G,L)$ is as small as possible.

\par

By Lemma~\ref{lem: isolated} we know that $\sigma(G,L) > 0$.  This means that there is a $c \in \mathcal{L}$ such that $\delta(L,c) > 0$.  Let $z_1, z_2, \ldots, z_{k-1}$ be new colors such that $\{z_1, \ldots, z_{k-1} \} \cap \mathcal{L} = \emptyset$.  We form the graph $G'$ from $G$ by adding $k$ isolated vertices to $G$ called: $v_1, v_2, \ldots, v_k$.  Clearly, the components of $G'$ have order at most $k$.  Now, let $L'$ be the $k$-assignment for $G'$ given by $L'(v)=L(v)$ when $v \in V(G)$ and $L'(v_i) = \{z_1, z_2, \ldots, z_{k-1}, c \}$ for $i = 1, \ldots, k$.

\par

Let $\mathcal{L}'$ be the palette of colors associated with $L'$.  Clearly, $\mathcal{L}' = \mathcal{L} \cup \{z_1, \ldots, z_{k-1} \}$.  Note that for each $x \in \mathcal{L} - \{c \}$, $\delta(L',x) = \delta(L,x)$.  Moreover, $\delta(L',z_i) = -k < 0$ for each $i= 1, \ldots, k-1$.  Finally, $\delta(L',c) = \delta(L,c) - k < \delta(L,c).$  It follows that $\sigma(G',L') < \sigma(G, L)$.  Thus, there is a proportional $L'$-coloring, $f$, of $G'$.  Since $\eta_{L'}(z_i)=k$ for each $i=1, \ldots, k-1$, we have that $f(\{v_1, \ldots, v_k \}) = \{z_1, z_2, \ldots, z_{k-1}, c \}.$  It follows that restricting $f$ to $V(G)$ yields a proportional $L$-coloring of $G$ which is a contradiction.
\end{proof}

\begin{customcor} {\bf \ref{cor: fullcliques}}
Suppose that $G$ is a disjoint union of cliques and the largest component of $G$ has $t$ vertices.  Then, $\chi_{pc}(G)=t$.
\end{customcor}

\begin{proof}
Since $G$ has a component isomorphic to $K_t$, Proposition~\ref{lem: monotone} and the fact that $K_t$ is not even $t-1$-colorable implies that $\chi_{pc}(G) > t-1$.  Theorem~\ref{thm: fullcomponents} and the fact that $G$ has no components of order greater than $t$ implies that $G$ is proportionally $t$-choosable.
\end{proof}

\subsection{Proportional Choosability of Stars}  In this subsection we completely determine the proportional choosability of stars by proving Theorem~\ref{thm: fullstars}.

\begin{customthm} {\bf \ref{thm: fullstars}}
$\chi_{pc}(K_{1,m}) = 1 + \lceil m/2 \rceil$.
\end{customthm}

\begin{proof}
Throughout, let $G=K_{1,m}$, with $v$ as the central vertex and $Z$ the set of leaves.

By Corollary~\ref{cor: maxdegreesec}, we know that
$$\chi_{pc}(G) > \frac{1 + \Delta(G)}{2} = \frac{1}{2} + \frac{m}{2} \geq \left \lceil \frac{m}{2} \right \rceil.$$
So, we have that $\chi_{pc}(G) \geq 1 + \lceil m/2 \rceil$.

Now, suppose that $k = 1 + \lceil m/2 \rceil$, and $L$ is an arbitrary $k$-assignment for $G$.  We will show that there is a proportional $L$-coloring for $G$.  By Propositions~\ref{pro: orderplus1}, we may assume that $k \leq m-1$.  So, $1 + m/2 \leq k \leq m-1$ which means $m \geq 4$ and $k \geq 3$.  We will prove the desired by considering 2 cases: (1)~all colors in $L(v)$ have multiplicity at least $k+1$ and (2)~there exists $a \in L(v)$ such that $\eta(a) \leq k$.

For case~(1), suppose that $L(v) = \{c_1, c_2, \ldots, c_k \}$.  We begin by constructing a proper $L$-coloring of $G$ as follows.  We color $v$ with $c_1$.  Then, since each color in $L(v)$ appears in at least $k$ lists corresponding to the vertices in $Z$, we may color 2 vertices in $Z$, say $z_1$ and $z_2$, with $c_2$.  We then proceed inductively.  Specifically, for each $r= 3, 4, \ldots, k$, there are at least $k-r+1$ lists corresponding to vertices in $Z - \{z_1, z_2, \ldots, z_{r-1} \}$ that contain $c_r$.  So, we may color a vertex in $Z - \{z_1, z_2, \ldots, z_{r-1} \}$, say $z_r$, with $c_r$.  Now, notice we have used each color in $L(v) - \{c_2 \}$ once, and we have used $c_2$ twice.  Moreover, we have colored $k$ vertices in $Z$.  So, there are $m-k$ vertices in $Z$ that are yet to be colored.  Let $Z'$ be the set of uncolored vertices in $Z$.  We note that:
$$1 \leq m - k \leq m - \left ( \frac{m}{2} + 1 \right ) = \frac{m}{2} - 1 \leq k - 2$$
which means that $1 \leq |Z'| \leq k - 2.$  Now, for each $z \in Z'$, we let $L'(z) = L(z) - \{c_1, c_2 \}$.  Notice $|L'(z)| \geq k-2$ for each $z \in Z'$.  Since $|Z'| \leq k-2$, it is possible to greedily color the vertices in $Z'$ with $|Z'|$ pairwise distinct colors.  Notice that the resulting coloring, $f$, is a proper $L$-coloring for $G$, which uses no color more than twice.  Furthermore, the only colors used twice by $f$ must be in $L(v) - \{c_1 \}$.  This means that for each $c \in \mathcal{L}$, $|f^{-1}(c)| \leq \lceil \eta(c)/k \rceil$ (i.e.\ $f$ uses no color in $\mathcal{L}$ excessively).  Since $\max_{c \in \mathcal{L}} \eta(c) \leq m+1 < m+2 \leq 2k$, Lemma~\ref{lem: algorithm} implies that there is a proportional $L$-coloring for $G$.

For case~(2), take an arbitrary huing of $L$.  By Lemma~\ref{lem:hues}, there is a proportional $L$-labelling $f$ of $G$ with $f(v)=a$.  Since $\eta(a)\leq k$, no other vertex receives color $a$, and so this coloring is a proportional $L$-coloring of $G$.
\end{proof}

\end{document}